\documentclass[reqno]{article}
\usepackage{amsmath,amsfonts,amsthm}

\newtheorem{theorem}{Theorem}
\newtheorem{lemma}[theorem]{Lemma}
\newtheorem{corollary}[theorem]{Corollary}

\begin{document}

\title{Judicious partitions of 3-uniform hypergraphs}
\author{John Haslegrave
\thanks{Trinity College, Cambridge CB2 1TQ, UK}
}

\maketitle

\begin{abstract}
The vertices of any graph with $m$ edges can be partitioned into two parts 
so that each part
meets at least $\frac{2m}{3}$ edges. Bollob\'as and Thomason conjectured 
that the vertices of
any $r$-uniform graph may be likewise partitioned into $r$ classes such 
that each part meets at
least $cm$ edges, with $c=\frac{r}{2r-1}$. In this paper, we prove this 
conjecture for the case $r=3$.
In the course of the proof we shall also prove an extension of the graph 
case which was conjectured by Bollob\'as and Scott.
\end{abstract}

\section{Introduction}
Given a graph $G$, it is easy to find a bipartition $V(G)=V_1\cup V_2$
such that at least half of the edges in $G$ join $V_1$ to $V_2$. It is
only slightly less trivial to find a bipartition $V_1\cup V_2$ such that 
each of $V_1$ and $V_2$ meets at least $2/3$ of the edges;
equivalently, each class in the bipartition contains at most $1/3$
of the edges (see, e.g., \cite{MGT}). (In fact, as $e(G) \to \infty$,
it is shown in \cite{BS99} that there is a bipartition in which each class
contains not much more than $1/4$ of the edges, which is trivially 
best possible by considering complete graphs.) 
These two ways of formulating the problem are equivalent for partitions into 
two parts, but give rise to
different generalisations for more parts. In this paper, we shall be 
concerned only with the problem of
meeting many edges; the problem of spanning few edges is addressed in 
\cite{BS99} 
for the graph case and \cite{BS97}  for the hypergraph case.

A particularly interesting case occurs when we partition the vertices
of an $r$-uniform hypergraph into $r$ classes, so that an edge may meet
every class. Bollob\'as and Thomason (see \cite{BRT93}, \cite{BS00}) 
conjectured that every $r$-uniform hypergraph with $m$ edges has an
$r$-partition in which each class meets at least $\frac{r}{2r-1}m$
edges. In \cite{BS00}, Bollob\'as and Scott prove a bound of
$0.27 m$; for $r=3$ they claim the better bound $(5m-1)/9$, but
there is a gap in their proof.

In this paper we prove the Bollob\'as--Thomason conjecture in the case 
$r=3$; we also prove a conjecture of Bollob\'as and Scott \cite{BS00}.

\section{Good partitions}\label{secmain}
Suppose we are given a 3-uniform hypergraph $G$ on vertex set $V$ with 
$m$ edges. For subsets 
$A, B, C$ of $V$, write $d(A)$ for the number of edges of $G$ meeting $A$, 
and $e(A,B,C)$ for the number of
(distinct) edges of the form $\{a,b,c\}$ with $a\in A$, $b\in B$ and $c\in C$.
Also, define  the {\em degree of $(A,B,C)$} as $d(A,B,C)=d(A)+d(B)+d(C)$.
Much of the time our triple $(A,B,C)$ will be a partition
of $V$. We shall sometimes abuse this 
notation by writing $a$ for $\{a\}$. Also, we write $d_2(A)$ for the number of 
edges meeting $A$ in at least 2 vertices. As a shorthand, we call
a partition of the vertex set $V$ a {\em partition of the graph $G$}.

For $0<\varepsilon < 2/3$, we call a set of vertices 
{\em $\varepsilon$-good}  if it 
meets at least $\left(\frac{2}{3}-\varepsilon\right)m$ of the edges;
otherwise we call it  {\em $\varepsilon$-bad}. As expected, we say that a set
is {\em minimal $\varepsilon$-good} if it is $\varepsilon$-good
and every proper subset of it is $\varepsilon$-bad.

We shall deduce our main theorem from the following
somewhat technical result.

\begin{theorem}\label{engine}
Let  $\varepsilon\ge\frac{1}{15}$, and let
$G$ be a 3-uniform hypergraph which cannot be partitioned into three
$\varepsilon$-good sets. Then there is a partition $V=A\cup B\cup C$
such that $A$ and $B$ are minimal $\varepsilon$-good sets and
\begin{equation}
d(A, B, C)>(2+3\varepsilon)m \label{engines}
\end{equation}
\end{theorem}

Our proof of Theorem~\ref{engine} is based on two lemmas. 
In order to reduce the clutter, we 
call a partition  $V=A\cup B\cup C$ {\em
optimal} if its degree  $d(A, B, C)$ is as large as possible, 
{\em locally optimal} if this degree cannot be increased by moving a single 
vertex from one class to another, 
and {\em semi-optimal} if it cannot increased by
moving a vertex into $C$. Note that semi-optimality depends on the order
of the sets in our partition; we shall always take the last set, $C$, to
be the exceptional set.
Trivially, every optimal partition is
locally optimal and every locally optimal partition is semi-optimal; 
however, the degree of a locally optimal partition can be rather small.

A simple random argument shows that there is a partition with
$d(A, B, C)\ge \frac{19m}{9}$; however, if 
$V=\{v_{ij}: 1\le i,j \le 3\}$, the edges are
$\{v_{i1},v_{i2},v_{i3}\}$ and $\{v_{1i},v_{2i},v_{3i}\}$ for each $i$,
then for $V_i=\{v_{i1},v_{i2},v_{i3}\}$ the partition 
$V=V_1\cup V_2 \cup V_3$
is locally optimal and $d(V_1, V_2, V_3)$ is only $2m$. 
As we shall see, however, this example is
in some sense typical; locally optimal partitions for 
which $d(A, B, C)$ is close to $2m$ have $d(A),d(B),d(C)$ roughly equal.

\begin{lemma}\label{prep}
Let $V=A\cup B\cup C$ be a semi-optimal partition. Then
\begin{equation}\label{acond}
3e(A,A,A)+2e(A,A,B)\le e(A,B,C)+e(A,C,C),
\end{equation}
and, similarly,
\begin{equation}\label{bcond}
3e(B,B,B)+2e(A,B,B)\le e(A,B,C)+e(B,C,C).
\end{equation}
\end{lemma}
\begin{proof}
Clearly, it suffices to prove \eqref{acond}.
Pick a vertex $a\in A$, and let us see how the degree $d(A,B,C)$ 
changes if we move $a$ from $A$ to $C$.
Letting $A'=A\setminus\{a\}$, $C'=C\cup\{a\}$, 
the difference $d(A)-d(A')$ is the number of edges
meeting $A$ only in $a$, i.e.
\begin{equation*}
d(A)-d(A')=e(a,B,B)+e(a,B,C)+e(a,C,C).
\end{equation*}
Similarly, $d(C')-d(C)$ is the number of edges which 
contain $a$ but are disjoint from $C$:
\begin{equation*}
d(C')-d(C)=e(a,A,A)+e(a,A,B)+e(a,B,B).
\end{equation*}
The difference of these identities gives us
\begin{equation*}
d(A, B, C)+e(a,A,A)+e(a,A,B)=d(A', B, C')+e(a,B,C)+e(a,C,C).
\end{equation*}
As $(A,B,C)$ is a semi-optimal partition,
\begin{equation*}
d(A, B, C)\ge d(A', B, C'),
\end{equation*}
and so we find that
\begin{equation}\label{extra}
e(a,A,A)+e(a,A,B)\le e(a,B,C)+e(a,C,C).
\end{equation}
Since this holds for every $a\in A$,
\begin{equation*}
\sum_{a\in A}e(a,A,A)+\sum_{a\in A}e(a,A,B) 
\le \sum_{a\in A}e(a,B,C)+\sum_{a\in A}e(a,C,C),
\end{equation*}
and so
\begin{equation*}
3e(A,A,A)+2e(A,A,B) \le e(A,B,C)+e(A,C,C).
\end{equation*}
as required.
\end{proof}

Lemma~\ref{prep} tells us that, {\em a fortiori}, every optimal partition 
$(A, B, C)$ satisfies \eqref{acond} and \eqref{bcond}.
Next, we show that the semi-optimality of a partition $(A, B, C)$ is
preserved if we move vertices into $C$.

\begin{lemma}\label{reduce}
Let $(A,B,C)$ be a semi-optimal partition, $(A', B', C')$ be a partition
with $A'\subseteq A$, $B'\subseteq B$ and $C' \supseteq C$.
Then $(A', B', C')$ is also semi-optimal.
\end{lemma}
\begin{proof}
It suffices to show that if $a\in A'$,  $A''=A'\setminus \{a\}$ and 
$C''=C'\cup \{a\}$ then $(A'',B',C'')$ is semi-optimal. Clearly,
\begin{equation*}
d(A')-d(A'')=e(a,B',B')+e(a,B',C')+e(a,C',C')
\end{equation*}
and
\begin{equation*}
d(C'')-d(C')=e(a,A',A')+e(a,A',B')+e(a,B',B').
\end{equation*}
Taking the difference of these identities, we find that
\begin{eqnarray}
 d(A', B', C') + && \hspace{-15pt} e(a,A',A')\ +  e(a,A',B') \nonumber \\
&& \hspace{-27pt} = \ d(A'', B', C'')+e(a,B',C')+e(a,C',C').  \label{splodge}
\end{eqnarray}
Since $(A,B,C)$ is semi-optimal and $a\in A'\subseteq A$, 
inequality \eqref{extra} holds. Consequently, as
$A'\subseteq A, B'\subseteq B, B'\cup C'\supseteq B\cup C$, 
we find that
\begin{eqnarray*}
 e(a,A',A')+e(a,A',B')&\le& e(a,A,A)+e(a,A,B)\\
 &\le& e(a,B,C)+e(a,C,C) \\
 &=& e(a,B\cup C,B\cup C)-e(a,B,B) \\
 &\le& e(a,B'\cup C',B'\cup C')-e(a,B',B') \\
 &=& e(a,B',C')+e(a,C',C').
\end{eqnarray*}
Together with (\ref{splodge}), this gives
\begin{equation*}
d(A', B', C') \ge d(A'', B', C''),
\end{equation*}
as required.
\end{proof}

After this preparation, we are ready to prove Theorem 1.
\begin{proof}
Let $V=A\cup B\cup C$ be an optimal partition; renaming the parts, 
if necessary, we may assume that 
$d(C)\le d(A), d(B)$. Since, by assumption, at least one of these classes
is $\varepsilon$-bad, we must have 
\begin{equation}
d(C)<\left(\frac{2}{3}-\varepsilon\right)m. \label{csmall}
\end{equation}
Since the partition $(A, B, C)$ is semi-optimal (in fact, optimal),
inequalities \eqref{acond} and \eqref{bcond} hold.

We claim that \eqref{csmall}, \eqref{acond} and \eqref{bcond} together 
imply \eqref{engines}.
To prove this claim, add \eqref{acond} and \eqref{bcond}, and then 
add $e(C,C,C)$ to both sides to obtain
\begin{eqnarray}
2\big(e(A,A,A)+ \hspace{-18pt} && e(B,B,B)+e(A,A,B)+e(A,B,B)\big) \nonumber \\
&&\hspace{-10pt}+e(A,A,A)+e(B,B,B)+e(C,C,C)  \nonumber \\
&&\hspace{-10pt}\le 2e(A,B,C)+e(A,C,C)+e(B,B,C)+e(C,C,C).  \label{long}
\end{eqnarray}
Note that
\begin{equation}
e(A,A,A)+e(B,B,B)+e(A,A,B)+e(A,B,B)=m-d(C) \label{ida}
\end{equation}
and
\begin{equation}
e(A,B,C)+e(A,C,C)+e(B,B,C)+e(C,C,C)\le d(C). \label{idb}
\end{equation}
Substituting (\ref{ida}) and (\ref{idb}) into (\ref{long}) gives
\begin{equation*}
 2\big(m-d(C')\big)+e(A,A,A)+e(B,B,B)+e(C,C,C)
 \le e(A,B,C)+d(C). 
\end{equation*}
Recalling (\ref{csmall}), we see that
\begin{equation}
e(A,B,C)-\left(e(A,A,A)+e(B,B,B)+e(C,C,C)\right) 
> 3\varepsilon.m . \label{almost}
\end{equation}
We may regard $d(A, B, C)$ as the sum over the {\em edges} of the number 
of parts meeting that edge;
thus, since $e(A,B,C)$ edges meet three parts, $e(A,A,A)+e(B,B,B)+e(C,C,C)$ 
meet one part and the others
meet two, we have
\begin{equation*}
d(A)+d(B)+d(C) = 2m+e(A,B,C)-\big(e(A,A,A)+e(B,B,B)+e(C,C,C)\big).
\end{equation*}
Together with (\ref{almost}), this gives us \eqref{engines}:
\begin{equation*}
d(A, B, C)>(2+3\varepsilon)m,
\end{equation*}
proving our claim.

We shall now find a new partition which satisfies the other 
requirements of Theorem 1.
Notice that \eqref{engines} and \eqref{csmall} 
give $d(A)+d(B)>\left(\frac{4}{3}+4\varepsilon\right)m$.
Since $d(A), d(B)\le m$, we must have 
$d(A), d(B)>\left(\frac{1}{3}+4\varepsilon\right)m
\ge\left(\frac{2}{3}-\varepsilon\right)m$ since
$\varepsilon\ge\frac{1}{15}$.

Let $A'\subseteq A$ be a minimal subset with 
$d(A)\ge\left(\frac{2}{3}-\varepsilon\right)m$, and let $B'$
be defined similarly. Let $C'=C\cup(A\setminus A')\cup(B\setminus B')$. 
By Lemma \ref{reduce}, $A', B', C'$
still satisfies the conditions of Lemma \ref{prep} so
\begin{equation*}
3e(A',A',A')+2e(A',A',B')\le e(A',B',C')+e(A',C',C')
\end{equation*}
and
\begin{equation*}
3e(B',B',B')+2e(A',B',B')\le e(A',B',C')+e(B',C',C').
\end{equation*}
Also, by the assumption that no partition into $\varepsilon$-good parts exists,
\begin{displaymath}
d(C')<\left(\frac{2}{3}-\varepsilon\right)m,
\end{displaymath}
and, as above, we still have
\begin{displaymath}
d(A')+d(B')+d(C')>(2+3\varepsilon)m.
\end{displaymath}
Since, by construction, $A',B'$ are minimal $\varepsilon$-good sets, 
$A',B',C'$ is of the form required.
\end{proof}

Note that we use the fact that $\varepsilon\ge\frac{1}{15}$ only to 
deduce that any locally optimal partition
has at least two $\varepsilon$-good parts; the value of 
$\frac{1}{15}$ is in fact tight. To see this consider the
hypergraph with vertex set $\{a,b,c,d,e,f,g\}$ and edges 
$abc, def, adg, beg, cfg$. A locally optimal partition
is $A=\{a,d,g\}, B=\{b,e\}, C=\{c,f\}$, but 
$d(B)=d(C)=3=\left(\frac{2}{3}-\frac{1}{15}\right)m$.

We can immediately deduce a partial result from Theorem \ref{engine}.
\begin{corollary}\label{fivenine}
There exists a partition with each part meeting at least $\frac{5m}{9}$ edges.
\end{corollary}
\begin{proof}
Suppose not. Let $(A,B,C)$ be the partition guaranteed by
Theorem \ref{engine} for $\varepsilon=\frac{1}{9}$.
Since  every edge not meeting $C$
meets either $A$ or $B$ in at least 2 vertices, we have
\begin{displaymath}
d(A)+d(B)+d_2(A)+d_2(B)>\left(\frac{4}{3}+4\varepsilon\right)m
+m-d(C)>\frac{20m}{9},
\end{displaymath}
where, as before, $d_2(A)$ denotes the number of edges meeting $A$ 
in at least 2 vertices,

We may assume that  
$d(A)+d_2(A)>\frac{10m}{9}$. Write $d(A)=\frac{5m}{9}+\alpha$ so that 
$d_2(A)>\frac{5m}{9}-\alpha$. Trivially, 
$d_2(A)>0$, so $A$ must contain at least 2 vertices. Pick $a\in A$; since, by 
Theorem \ref{engine}, $A$ is a minimal $\varepsilon$-good set, $d(A\setminus\{a\})<\frac{5m}{9}$, so there are at 
least $\alpha$ edges meeting $A$ only in $a$. 
Pick a different vertex $a'\in A$; if $e$ is an edge which either
meets $A$ in at least two vertices or meets $A$ only in $a$, the edge
$e$ meets $A\setminus\{a'\}$, and so
\begin{displaymath}
d(A\setminus\{a'\})>\alpha+d_2(A)>\frac{5m}{9}
\end{displaymath}
contradicting minimality of $A$.
\end{proof}

\section{Multigraphs with special vertices}
To go further we shall be more careful with our estimates. 
To this end, we shall prove a judicious partitioning 
result about (multi-)graphs with some ``special'' vertices. 
We may think of these as (multi-)hypergraphs with 
edges of size at most 2. 
However, we use the formulation of special vertices to highlight 
the vital point that we permit repeated edges of size 2, 
but do not permit repeated edges of size 1. This result, as well 
as more general results in a similar vein, were conjectured in \cite{BS02a}.

For such a multigraph $G=(V,E,S)$ on vertex set $V$ with special vertices 
$S \subseteq V$,
and sets $W_1,W_2 \subseteq V$ of vertices, we write
$f(W_i)$ for the number of special vertices in $W_i$. Also, as usual,
$e(W_1)$ denotes the number of edges 
spanned by $W_1$ and $e(W_1,W_2)$ the number of 
edges $xy$ of the form $x\in W_1$, $y\in W_2$.
\begin{theorem}\label{multi}
Let $G=(V,E,S)$ be a multigraph with $m$ edges 
and $k$ special vertices. Then there is a partition
$V=V_1\cup V_2$ such that, for $i=1,2$:
\begin{equation}
e(V_i)+f(V_i)\le \frac{m}{3}+\frac{k+1}{2}. \label{spec}
\end{equation}
\end{theorem}
\begin{proof}

Again, we call a partition {\em optimal} if it minimizes $e(V_1)+e(V_2)$, and 
{\em locally optimal} if if this sum cannot be increased by moving a single 
vertex from one class to the other. We first note that if $V_1, V_2$ is locally optimal
then $e(V_i)\le\frac{m}{3}$: indeed, for each vertex $v\in V_1$ we must have
\begin{displaymath}
e(v,V_1)\le e(v,V_2),
\end{displaymath}
and summing this over all $v\in V_1$ gives
\begin{equation}
2e(V_1)\le e(V_1,V_2). \label{that}
\end{equation}
Observing that $m=e(V_1)+e(V_2)+e(V_1,V_2)$, \eqref{that} is equivalent to
\begin{displaymath}
3e(V_1)+e(V_2)\le m.
\end{displaymath}

Suppose that no partition satisfying \eqref{spec} exists. Let us choose an optimal 
partition $V_1, V_2$ for which $|f(V_1)-f(V_2)|$ is as small as possible (among optimal partitions) 
and $f(V_1) \ge f(V_2)$. Since $V_1, V_2$ is optimal, $e(V_2)\le\frac{m}{3}$, and since 
$f(V_2)\le f(V_1)$, $f(V_2)\le \frac{k-1}{2}$; thus $V_2$ satisfies \eqref{spec}. By assumption, then,
$V_1$ cannot also satisfy \eqref{spec}, so
\begin{displaymath}
e(V_1)+f(V_1) > \frac{m}{3}+\frac{k+1}{2}.
\end{displaymath}
Since the partition is optimal, $e(V_1)\le\frac{m}{3}$ and so we must have
\begin{displaymath}
f(V_1) > \frac{k+1}{2},
\end{displaymath}
and since $f(V_1)+f(V_2)=m$, therefore
\begin{equation}
f(V_1) > f(V_2)+1. \label{gap}
\end{equation}
Let $v$ be any vertex in $V_1$; since $V_1,V_2$ is locally optimal we must have
\begin{equation}
e(v,V_1)\le e(v,V_2). \label{fg}
\end{equation}
If also $v$ is special, then by choice of our partition we must have
\begin{displaymath}
e(v,V_1)\ne e(v,V_2),
\end{displaymath}
and hence
\begin{equation}
e(v,V_1)+1\le e(v,V_2), \label{fgh}
\end{equation}
since otherwise moving $v$ into $V_2$ gives another optimal partition $V'_1, V'_2$, 
and using \eqref{gap},
\begin{eqnarray*}
|f(V'_1)-f(V'_2)| &=& |f(V_1)-f(V_2)-2| \\ 
&<& |f(V_1)-f(V_2)|.
\end{eqnarray*}

Now we aim to show that we may move vertices across from $V_1$ to $V_2$ 
to get a partition satisfying \eqref{spec}. Take $W_2\subseteq V_2$
maximal such that
\begin{displaymath}
e(W_2)+f(W_2)\le \frac{m}{3}+\frac{k+1}{2},
\end{displaymath}
and write $W_1=V\setminus W_2$. Now if $w\in W_1$ is special then by \eqref{fgh}
\begin{equation}
e(w,W_1)+1\le e(w,V_1)+1\le e(w,V_2)\le e(w,W_2), \label{jk}
\end{equation}
and if $w\in W_1$ is not special
\begin{equation}
e(w,W_1)\le e(w,V_1)\le e(w,V_2)\le e(w,W_2). \label{jkl}
\end{equation}
Equivalently to \eqref{jk} and \eqref{jkl} we may write for any $w\in W_1$
\begin{displaymath}
e(w,W_1)+{\mathbf 1}_{w\in S}\le e(w,W_2), 
\end{displaymath}
and sum over all $w\in W_1$ to give
\begin{displaymath}
2e(W_1)+f(W_1) \le e(W_1,W_2).
\end{displaymath}
Adding $e(W_1)+2f(W_1)$ to both sides gives
\begin{displaymath}
3e(W_1)+3f(W_1) \le e(W_1)+e(W_1,W_2)+2f(W_1),
\end{displaymath}
and since $m=e(W_1)+e(W_2)+e(W_1,W_2)$, we have
\begin{displaymath}
e(W_1)+f(W_1) \le \frac{1}{3}(m+2f(W_1)-e(W_2)).
\end{displaymath}

If $e(W_1)+f(W_1)\le \frac{m}{3}+\frac{k+1}{2}$ we are done. If not, certainly 
$\frac{2k}{3} > \frac{k+1}{2}$,  i.e. $k>3$, 
and since $e(W_1)\le e(V_1)\le \frac{m}{3}$,
we must have $f(W_1)>\frac{k+1}{2}>2$. Also
\begin{displaymath}
\frac{m}{3}+\frac{k+1}{2} < e(W_1)+f(W_1) \le\frac{1}{3}(m+2f(W_1)-e(W_2)),
\end{displaymath}
and, since $f(W_1)+f(W_2)=k$,
\begin{displaymath}
\frac{f(W_1)+f(W_2)+1}{2} < \frac{2f(W_1)-e(W_2)}{3},
\end{displaymath}
{\em i.e.}
\begin{displaymath}
2e(W_2)+3f(W_2)+3 < f(W_1),
\end{displaymath}
and since $f(W_2)\ge 0$,
\begin{eqnarray*}
e(W_2)+f(W_2) &<& \frac{f(W_1)-3}{2} \\
&\le& \frac{k+1}{2}-2.
\end{eqnarray*}

Recall that $W_2$ is a maximal set satisfying
\begin{displaymath}
e(W_2)+f(W_2)\le \frac{m}{3}+\frac{k+1}{2}.
\end{displaymath}
If $w\in W_1$,
\begin{eqnarray*}
e(W_2\cup\{w\})+f(W_2\cup\{w\})&\le& e(W_2)+e(w,W_2)+f(W_2)+1 \\
&<&\frac{k+1}{2}+e(w,W_2).
\end{eqnarray*}
But $e(W_2\cup\{w\})+f(W_2\cup\{w\})
>\frac{m}{3}+\frac{k+1}{2}$, so $e(w,W_2)>\frac{m}{3}$ for each $w\in W_1$,
which is impossible since $|W_1| \ge f(W_1) \ge 3$.
\end{proof}
Theorem \ref{multi} tells us that if $G$ is a (multi-)hypergraph 
with $k$ distinct edges of size 1 and $m$ (not necessarily
distinct) edges of size at least 2, we may find a partition of $G$ 
into two parts $V_1,V_2$ which satisfies
\begin{displaymath}
d(V_i)\ge \frac{2m}{3}+\frac{k-1}{2}.
\end{displaymath}
We may see this by first replacing each edge of size greater than 2 
with a subedge of size 2. This is a strengthening
of the result proved in \cite{BS00} that we can acheive
\begin{displaymath}
d(V_i)\ge \frac{2m}{3}+\frac{k-1}{3}.
\end{displaymath}
Since we may need to apply \eqref{multi} to a hypergraph with repeated edges, 
it is vital we ensure there are no repeated
edges of size 1. 
As usual, we define the {\em restriction of a hypergraph $G$ to a subset
$U$ of its vertices} as the  multi-hypergraph with vertex set $U$, in which
the multiplicity of an  edge $e$ is $|{f\in E(G):f\cup U=e}|$.

We shall check that the number of vertices in $V\setminus U$ is at 
most 2 before restricting $G$ to $U$; since we
shall always start from a 3-uniform $G$ this will ensure no repeated edges 
of size 1. 

\section{The bound $c=\frac{3}{5}$}
We shall first give the basic argument, which proves the conjecture 
for hypergraphs which are above a certain size. We shall
then need to take more care in the details for small hypergraphs.
\begin{theorem}\label{211}
Let $G$ be a 3-uniform hypergraph with $m\ge 211$ edges. 
Then there exists a partition into three parts with each
part meeting at least $\frac{3}{5}m$ edges.
\end{theorem}
\begin{proof}
Suppose $G$ does not admit such a partition. 
Take $\varepsilon=\frac{1}{15}$ in Theorem \ref{engine}, 
and let $A,B,C$ be the partition
guaranteed. Then, by the result of the theorem,
\begin{displaymath}
d(A)+d(B)>\left(\frac{4}{3}+4\varepsilon\right)m=\frac{8}{5}m
\end{displaymath}
and $A,B$ are minimal $\varepsilon$-good sets, so no proper subset 
of $A$ or $B$ meets $\frac{3}{5}m$ edges.
Suppose without loss of generality that $d(A)\ge d(B)$. 
We distinguish two cases.

\vspace{5pt}
\textbf{Case 1.} $d(A)\ge\frac{9}{10}m$. 
By the minimality of $A$, for each vertex $a\in A$ there are more than
$\frac{3}{10}m$ edges meeting $A$ only at $a$. 
Hence any two vertices of $A$ between them meet more than 
$\frac{3}{5}m$ edges, and so cannot be a proper subset. Hence $|A|\le 2$. 

Suppose $|A|=2$ and $d_2(A)\ge\frac{3}{10}m$. 
Then there are at least $\frac{3}{10}m$ edges which meet $A$ only at $a_1$, 
and at least $\frac{3}{10}m$ edges which meet $A$ in both vertices, 
totalling $\frac{3}{5}m$ edges which meet $a_2$.
Since $A$ is minimal, this is impossible, so fewer than $\frac{3}{10}m$ 
edges meet $A$ in both vertices.

Now we consider $H$, the restriction of $G$ to $B\cup C$ 
(recall that we defined this to be a multi-hypergraph), 
and note that as $|A|\le 2$, $e(H)=e(G)$ and there can be no repeated 
edges of size 1; also, since fewer than $\frac{3}{10}m$
edges meet $A$ in two vertices, there are $k<\frac{3}{10}m$ edges of 
size 1 in $H$. Applying the result of
Theorem \ref{multi}, we may find a partition $D_1\cup D_2=B\cup C$ with
\begin{displaymath}
d_H(D_i)\ge\frac{2m-2k}{3}+\frac{k-1}{2}=\frac{2m}{3}-\frac{k+3}{6}
\end{displaymath}
and $k<\frac{3}{10}m$, so
\begin{displaymath}
d_H(D_i)\ge\frac{2m-2k}{3}-\frac{1}{2}-\frac{m}{20}
=\frac{3m}{5}+\frac{m}{60}-\frac{1}{2}.
\end{displaymath}
Since $m\ge30$, $A,D_1,D_2$ form a suitable partition.

\vspace{5pt}
\textbf{Case 2.} $\frac{9}{10}m\ge d(A)\ge\frac{8}{10}m$, 
so $d(B)\ge\frac{7}{10}m$. Then for each $a\in A$ there must be
more than $\frac{2}{10}m$ edges meeting $A$ only at $a$, 
and so any three vertices in $A$ between them meet more than $\frac{3}{5}m$
edges, so $|A|\le 3$. Similarly, $|B|\le 6$. 
Since $|A\cup B\le 9$, there can be at most $\binom{9}{3} = 84$ 
edges which do not meet
$C$, so $\frac{2}{5}m < 84$, which is false if $m\ge 211$. 
\end{proof}

\section{Small $m$}

In this section we will show the same result for all hypergraphs, not just those with $m\ge 211$. In order to do this, 
we shall be careful to take into account that intermediate expressions must be integer valued. Throughout, $G$ shall
be a 3-uniform hypergraph on vertex set $V$ with $m$ edges; we shall call a set {\em good} if it meets at least $\frac{3}{5}m$ edges, and
{\em bad} otherwise.

The following result is immediate, but we shall refer to it more than once.

\begin{lemma}\label{degree}
If there is a set $A\subset V$ with $d_2(A)=0$ and $d(A)\ge\frac{3}{5}m$ then there exists a partition of $V$ into three good sets. 
In particular, if $\Delta(G)\ge\frac{3}{5}m$ then such a partition exists.
\end{lemma}
\begin{proof}
Suppose $d(A)\ge\frac{3}{5}m$. For each edge $e$ of $G$ pick a subedge of size 2 which does not meet $A$; since 
no edge contains more than one vertex in $A$, this is possible. This gives a multigraph $H$ on $V\setminus A$ with 
$m$ edges. By the result of (Theorem \ref{multi}), we can partition $V\setminus A$ into two parts $B$, $C$, each 
meeting at least $\frac{2}{3}m$ edges of $H$. 
\end{proof}

We shall also use another similar result.

\begin{lemma}\label{twodeg}
If $m\ge 10$ and $G$ has two vertices of degree $\lceil\frac{3}{5}m\rceil -1$, then there is a partition of $V$ into three good sets.
\end{lemma}
\begin{proof}
Let $a$, $b$ be two such vertices. There is an edge which does not contain $a$; let $c$ be a vertex in that edge other than $b$. 
There is an edge which does not contain $b$, let $d$ be a vertex in that edge other than $a,c$. $\{a,c\}$ and $\{b,d\}$ each 
meet at least $\lceil\frac{3}{5}\rceil$ edges, so are good. 
$V\setminus \{a,b,c,d\}$ meets at least $m-4$ edges, as there are only four possible edges 
contained in $\{a,b,c,d\}$; since $m\ge10$, $m-4\ge\lceil\frac{3}{5}\rceil$, so $V\setminus \{a,b,c,d\}$ is also good.
\end{proof}

Suppose $V$ cannot be partitioned into three good sets. Then every partition has some part meeting 
at most $\left\lceil\frac{3}{5}m\right\rceil-1$ edges. 
Let $\delta=\frac{2}{3}-\frac{\lceil\frac{3}{5}m\rceil-1}{m}>\frac{2}{3}-\frac{3}{5}=\frac{1}{15}$. 
For any $\frac{1}{15}<\varepsilon<\delta$, $G$ has no tripartition into $\varepsilon$-good parts, 
and so the partition guaranteed by Theorem \ref{engine} has $d(A)+d(B)+d(C)>(2+3\varepsilon)m$. 
If there is no such partition with $d(A)+d(B)+d(C)\ge(2+3\delta)m$, then by taking
$\varepsilon$ suitably close to $\delta$ we obtain a contradiction. Hence Theorem \ref{engine} implies

\begin{corollary}\label{discrete}
Let $G$ be a 3-uniform hypergraph which cannot be partitioned into three
good sets. Then there is a partition $V=A\cup B\cup C$
such that $A$ and $B$ are minimal good sets and
\begin{displaymath}
d(A)+d(B)+d(C)\ge(2+3\delta)m=4m-3(\lceil\frac{3}{5}m\rceil-1)
\end{displaymath}
and since $C$ is bad,
\begin{eqnarray*}
d(A)+d(B)&\ge&(2+3\delta)m \\
&=&4m-4(\lceil\frac{3}{5}m\rceil-1) \label{many} \\
&>&\frac{8}{5}m \label{clean}
\end{eqnarray*}
\end{corollary}

We shall now be more careful about bounding the sizes of the sets $A$, $B$.

\begin{lemma}
In the partition guaranteed by Corollary \ref{discrete}, either $|A|=2$ or $|B|=2$.
\end{lemma}
\begin{proof}
Since any edge which does not meet $C$ meets either $A$ or $B$ in at least two places, 
\begin{displaymath}
d_2(A)+d_2(B)\ge m-d(C) >\frac{2}{5}m,
\end{displaymath}
and so, recalling \eqref{clean}, 
\begin{displaymath}
2d(A)+2d(B)+d_2(A)+d_2(B)>\frac{18}{5}m.
\end{displaymath}
Without loss of generality, assume $2d(A)+d_2(A)\ge 2d(B)+d_2(B)$, so that
\begin{displaymath}
2d(A)+d_2(A)>\frac{9}{5}m.
\end{displaymath}
Since $A$ is minimally good, for each $a\in A$ there are more than $d(A)-\frac{3}{5}m$ edges which meet $A$ only at $a$. There are also $d_2(A)$
edges meeting $A$ in more than one vertex, and so if $|A|\ge 3$
\begin{displaymath}
d(A)\ge 3(d(A)-\frac{3}{5}m)+d_2(A),
\end{displaymath}
contradicting the previous inequality. Thus $|A|\le 2$, and Lemma \ref{degree} imples that $|A|=2$.
\end{proof}

\begin{lemma}\label{upb}In the partition guaranteed by Corollary \ref{discrete}, reordering $A, B$ if necessary so that $|A|=2$,
\begin{displaymath}
d(A)\le 8\lceil\frac{3}{5}m\rceil -4m -5.
\end{displaymath}
\end{lemma}
\begin{proof}
Since $A$ is minimally good, for each $a\in A$ there are more than $d(A)-\frac{3}{5}m$, and so at least $d(A)-\lceil\frac{3}{5}m\rceil+1$, 
edges which meet $A$ only at $a$, and so
\begin{displaymath}
d(A)\ge 2(d(A)-\lceil\frac{3}{5}m\rceil+1)+d_2(A).
\end{displaymath}
If $d(A)\ge 8\lceil\frac{3}{5}m\rceil -4m -4$,
\begin{eqnarray*}
d_2(A)&\le& 2\lceil\frac{3}{5}m\rceil-d(A)-2 \\
&\le& 4m+2-6\lceil\frac{3}{5}m\rceil.
\end{eqnarray*}
Consider $H$, the multi-hypergraph obtained by restricting $G$ to $B\cup C$. 
Note that there can be no repeated edges of size 1 and at most $d_2(A)$ edges of size 1 in $H$. 
Writing $k$ for the number of such edges, and applying the result of Theorem \ref{multi}, 
we may find a partition $D_1\cup D_2=B\cup C$ with
\begin{displaymath}
d_H(D_i)\ge\frac{2m-2k}{3}+\frac{k-1}{2}=\frac{2m}{3}-\frac{k+3}{6},
\end{displaymath}
and $k\le 4m+2-6\lceil\frac{3}{5}m\rceil$, so
\begin{displaymath}
d_H(D_i)\ge\frac{2m}{3}-\frac{1}{2}-\frac{k}{6}=\lceil\frac{3}{5}m\rceil-\frac{5}{6}.
\end{displaymath}
Since $d_H(D_i)$ must be an integer, it is at least $\lceil\frac{3}{5}m\rceil$, so $A,D_1,D_2$
contradicts our assumption that $G$ does not have a good partition.
\end{proof}

We shall now bound $|B|$.
\begin{lemma}\label{upbb}
In the partition guaranteed by Corollary \ref{discrete}, reordering $A, B$ if necessary so that $|A|=2$, $|B|\le3$.
\end{lemma}
\begin{proof}
Suppose $|B|\ge4$. By Theorem \ref{discrete},
\begin{displaymath}
d(A)+d(B)\ge 4m-4(\lceil\frac{3}{5}m\rceil-1).
\end{displaymath}
Since $A$ is minimally good, for each $a\in A$ there at least $d(A)-\lceil\frac{3}{5}m\rceil+1$ edges which meet $A$ only at $a$, and so
\begin{equation}
d(A)\ge 2(d(A)-\lceil\frac{3}{5}m\rceil+1)+d_2(A). \label{zx}
\end{equation}
Similarly,
\begin{equation}
d(B)\ge 4(d(B)-\lceil\frac{3}{5}m\rceil+1)+d_2(B). \label{zxc}
\end{equation}
\eqref{zx} and \eqref{zxc} together give
\begin{eqnarray*}
d_2(A)+d_2(B)&\le& 6\lceil\frac{3}{5}m\rceil-d(A)-3d(B)-6 \\
&=& 6\lceil\frac{3}{5}m\rceil+2d(A)-3(d(A)-d(B))-6 \\
&\le& 18\lceil\frac{3}{5}m\rceil+2d(A)-12m-18.
\end{eqnarray*}
By Lemma \ref{upb}, therefore,
\begin{displaymath}
d_2(A)+d_2(B)\le 34\lceil\frac{3}{5}m\rceil-20m-28.
\end{displaymath}
Since any edge not meeting $C$ meets either $A$ or $B$ in at least two vertices,
\begin{eqnarray*}
d(C)-\lceil\frac{3}{5}m\rceil&\ge& m-d_2(A)-d_2(B)-\lceil\frac{3}{5}m\rceil \\
&\ge& 21m+28-35\lceil\frac{3}{5}m\rceil \\
&\ge& 21m+28-35\left(\frac{3}{5}m+\frac{4}{5}\right) =0.
\end{eqnarray*}
Again, this contradicts our assumption that $G$ does not have a good partition.
\end{proof}

We now have a much better bound than the one given by Theorem \ref{211}.

\begin{corollary}\label{large}
If $m\ge 25$, there is a partition of $V$ into three good sets.
\end{corollary}
\begin{proof}Suppose not. Lemmas \ref{upb} and \ref{upbb} then imply that there is a partition with $A$, $B$ being good sets
and $|A\cup B|\le 5$. But then at most 10 edges lie entirely within $A\cup B$, so $d(C)\ge m-10 \ge\frac{3}{5}m$
when $m\ge 25$.
\end{proof}

\begin{lemma}\label{rest}
If $m\le 24$ and $m\neq 7$, there is a partition of $V$ into three good sets.
\end{lemma}
\begin{proof}
If $m=1$ or $m=2$, finding such a partition is trivial, so we shall assume $m\ge 3$.

Assume no good partition exists; taking $A,B,C$ to be the partition guaranteed by Corollary \ref{discrete} with $A$, $B$ reordered if necessary so that
$|A|=2$, we shall now consider the two possible cases given by Lemma \ref{upbb}.

\vspace{5pt}
\textbf{Case 1.} $|B|=3$.

Write $A=\{a_1, a_2\}$, $B=\{b_1, b_2, b_3\}$, and for each $i$ let the number of edges meeting $B$ only at $b_i$ be $x_i$.
Now $d(B)=x_1+x_2+x_3+d_2(B)$ and for each $i$,
\begin{displaymath}
\lceil\frac{3}{5}m\rceil-1\ge d(B\setminus\{b_i\})=d_2(B)+x_1+x_2+x_3-x_i.
\end{displaymath}
Since $A$ contains only two vertices, and $B$ three, there are only three possible edges which
do not meet $C$ and contain at most one vertex of $B$, so
\begin{displaymath}
d_2(B)\ge m-d(C)-3 \ge m-\lceil\frac{3m}{5}\rceil-2.
\end{displaymath}
Since $B$ is minimally good,
\begin{eqnarray*}
\lceil\frac{3}{5}m\rceil-1&\ge& d_2(B)+x_2+x_3 \\
&\ge& m-\lceil\frac{3}{5}m\rceil-2+x_2+x_3,
\end{eqnarray*}
so
\begin{displaymath}
x_2+x_3 \le 2\lceil\frac{3}{5}m\rceil+1-m.
\end{displaymath}
Hence one of $x_2$ and $x_3$ (say $x_3$) is at most $\lceil\frac{3m}{5}\rceil-\lfloor\frac{m-1}{2}\rfloor$.
Since $B$ is minimally good, $d(B)-x_3 \le \lceil\frac{3m}{5}\rceil-1$, so
\begin{displaymath}
d(B) \le 2\lceil\frac{3m}{5}\rceil-\lfloor\frac{m-1}{2}\rfloor-1.
\end{displaymath}
Using \eqref{many},
\begin{displaymath}
d(A)+d(B)\ge 4m-4(\lceil\frac{3}{5}m\rceil-1),
\end{displaymath}
so
\begin{equation}
d(A)\ge 4m-6\lceil\frac{3}{5}m\rceil+5+\lceil\frac{m-1}{2}\rceil. \label{qw}
\end{equation}
Comparing this with the upper bound obtained in Lemma \ref{upb}, we must have
\begin{equation}
14\lceil\frac{3}{5}m\rceil\ge10+8m+\lceil\frac{m-1}{2}\rceil, \label{qwe}
\end{equation}
and equality in \eqref{qwe} implies equality in \eqref{qw}. \eqref{qwe} is false for all 
values in the range $3\le m\le 24$ except 7, 12 and 17. For $m=12$ and $m=17$
there is equality in \eqref{qwe} and so also in \eqref{qw}. If $m=12$, therefore, we have
$d(A)=11$. Since $A$ is minimally good, each vertex in $A$ meets at most 7 edges, and so $d_2(A)\ge 3$.
Since $d(A)=d(a_1)+d(a_2)-d_2(A)$, we must have $d(a_1)=d(a_2)=7$, and by Lemma \ref{twodeg} we can find
a good partition. If $m=17$, similarly, we have $d(A)=15$, $d(a_i)\le 10$ and $d_2(A)\ge 5$. Again, this implies
that $d(a_1)=d(a_2)=10$, and by Lemma \ref{twodeg} there is a good partition.

\vspace{5pt}
\textbf{Case 2.} $|B|=2$.

In this case, since there are at most 4 edges contained in $A\cup B$, $d(C)\ge m-4$ and so there exists a good partition if $m\ge 10$. 

Write $B=\{b_1, b_2\}$, and for each $i$ let the number of edges meeting $B$ only at $b_i$ be $x_i$.
Now $d(B)=x_1+x_2+d_2(B)$ and so,
\begin{displaymath}
\lceil\frac{3}{5}m\rceil-1\ge d_2(B)+x_i.
\end{displaymath}
Since $|A|=|B|=2$, there are only two possible edges which
do not meet $C$ and contain at most one vertex of $B$. So
\begin{displaymath}
d_2(B)\ge m-d(C)-2 \ge m-\lceil\frac{3m}{5}\rceil-1.
\end{displaymath}
Since $B$ is minimally good, $d_2(B)+x_2 \le \lceil\frac{3m}{5}\rceil-1$, so
\begin{displaymath}
x_2 \le 2\lceil\frac{3}{5}m\rceil-m.
\end{displaymath}
Since $B$ is minimally good, $d(B)-x_2 \le \lceil\frac{3m}{5}\rceil-1$, so
\begin{displaymath}
d(B) \le 3\lceil\frac{3m}{5}\rceil-m-1
\end{displaymath}
Using \eqref{many},
\begin{displaymath}
d(A)+d(B)\ge 4m-4(\lceil\frac{3}{5}m\rceil-1),
\end{displaymath}
so
\begin{displaymath}
d(A)\ge 5m-7\lceil\frac{3}{5}m\rceil+5.
\end{displaymath}

Comparing this with the upper bound obtained in Lemma \ref{upb}, we must have
\begin{displaymath}
15\lceil\frac{3}{5}m\rceil\ge10+9m,
\end{displaymath}
which holds if and only if $m\equiv 2$ mod 5. Since $3\le m\le 9$, we must have $m=7$.
\end{proof}

\begin{lemma}\label{seven}
If $m=7$, there is a partition of $V$ into three good sets.
\end{lemma}
\begin{proof} 
Take $A,B,C$ to be the partition guaranteed by Corollary \ref{discrete}, reordering $A$, $B$ if necessary so that $d(A)\ge d(B)$. 
By \eqref{many}, $d(A)+d(B) \ge 12$, so either $d(A)=7$ or $d(A)=d(B)=6$.

\vspace{5pt}
\textbf{Case 1.} $d(A)=7$.
Since $A$ is minimally good, each vertex in $A$ meets at least three edges which do not contain any other vertex in $A$. 
Also, by Lemma \ref{degree}, some edge meets more than one vertex in $A$. Therefore $|A|=2$ and each vertex in $A$ meets 
exactly three edges which do not contain the other, and one edge which does.

Thus, writing $A=\{a,b\}$, $d(a)=d(b)=4$ and there is exactly one edge which contains both, $abc$, say. Choose any edge which 
does not contain $a$; this edge contains some vertex $d\neq b,c$. Choose any edge which does not contain $b$ and is not 
$acd$; this edge contains some vertex $e\neq a,c,d$. Now $\{a,d\}$ and $\{b,e\}$ are both good sets. 
Since neither $abd$ nor $abe$ is an edge of $G$, at most two edges do not meet $V\setminus \{a,b,d,e\}$, so this is also good 
and we have a good partition.

\vspace{5pt}
\textbf{Case 2.} $d(A)=d(B)=6$.
Relabel if necessary so that $d_2(A)\ge d_2(B)$. Since $d(C)\le 4$, and $d(C)\ge m-d_2(A)-d_2(B)$, $d_2(A)\ge 2$.
Since $A$ is minimally good, each vertex in $A$ meets at least two edges which do not contain any other vertex in $A$.
Therefore $|A|=2$ and each vertex in $A$ meets exactly two edges which do not contain the other, and two edges which do.

Thus, writing $A=\{a,b\}$, $d(a)=d(b)=4$ and there are exactly two edge which contains both, $abc$ and $abd$, say. 
There is one edge which meets neither $a$ nor $b$; this edge contains some vertex $e\neq c,d$. 
Since, by Lemma \ref{degree}, $\Delta(G)\le 4$, there must be at least 6 vertices of degree at least 1, so there exists
a vertex $f\neq a,b,c,d,e$ which meets an edge. Since $abf$ is not an edge of $G$, either there is an edge which meets 
$f$ but not $a$, or there exists an edge which meets $f$ but not $b$; assume without loss of generality the former.
Now $\{a,f\}$ and $\{b,e\}$ are both good sets. Since neither $abd$ nor $abe$ is an edge of $G$, at most two edges do 
not meet $V\setminus \{a,b,d,e\}$, so this is also good and we have a good partition.
\end{proof}

Corollary \ref{large}, together with Lemmas \ref{seven} and \ref{rest} give our main result.

\begin{theorem}
Let $G$ be a 3-uniform hypergraph with $m$ edges. 
Then there exists a partition into three parts with each
part meeting at least $\frac{3}{5}m$ edges.
\end{theorem}

\section{Final Remarks}
Two questions naturally arise from these results. Firstly, for $r=3$, can we acheive a better bound than $c=\frac{3}{5}$ for sufficiently large $m$?
It seems very likely that such a result is true; indeed, by analogy with the results of Bollob\'as and Scott in the graph case \cite{BS99},
we might expect there exist partitions with each part meeting at least $(\frac{19}{27}-o(1))m$ edges (considering complete hypergraphs 
shows that we cannot do better). However, we cannot hope to do better than $(\frac{2}{3}-o(1))m$ using these methods. 
The immediate difficulty in doing better than $\frac{3}{5}m$ is the need to find an optimal partition with at most one bad part in the 
proof of Theorem \ref{engine}, 

The second question is whether we can say anything for $r>3$. Again, the need to find a starting partition with at most one bad part is likely
to become much more difficult as the number of parts increases. A substantially different approach may well be needed to get a bound which 
remains good as $r$ becomes large.

\end{document}